\documentclass[reqno]{amsart}

\usepackage{graphicx,subfigure}

\numberwithin{equation}{section}

\usepackage[latin1]{inputenc}
\usepackage[english]{babel}

\usepackage{amsmath,amsthm,amsfonts,latexsym,amssymb}
\usepackage[colorlinks]{hyperref}
\hypersetup{linkcolor=blue,citecolor=blue,filecolor=black,urlcolor=blue}
\usepackage{comment}

\usepackage{color,amsthm,amsfonts}
\definecolor{darkgreen}{rgb}{0,0.7,0.1}

{ \theoremstyle{plain}
\newtheorem{theorem}{Theorem}[section]
\newtheorem{proposition}[theorem]{Proposition}
\newtheorem{lemma}[theorem]{Lemma}

  \theoremstyle{remark}
\newtheorem{remark}[theorem]{Remark}
  \theoremstyle{definition}

}

\begin{document}


\title[]{Multiplicity of solutions on a Nehari set in an invariant cone}
\author[F. Colasuonno]{Francesca Colasuonno}
\address{Francesca Colasuonno\newline\indent
Dipartimento di Matematica
\newline\indent
Alma Mater Studiorum Universit\`a di Bologna
\newline\indent
piazza di Porta San Donato 5, 40126 Bologna, Italy}
\email{francesca.colasuonno@unibo.it}

\author[B. Noris]{Benedetta Noris}

\author[G. Verzini]{Gianmaria Verzini}

\address{Benedetta Noris and Gianmaria Verzini\smallskip\newline\indent
Dipartimento di Matematica\newline\indent
Politecnico di Milano\newline\indent
Piazza Leonardo da Vinci 32, 20133 Milano, Italy}
\email{benedetta.noris@polimi.it}
\email{gianmaria.verzini@polimi.it}

\begin{abstract}
For $1<p<2$ and $q$ large, we prove the existence of two positive, nonconstant, radial and radially nondreacreasing solutions of the supercritical equation
\[-\Delta_p u+u^{p-1}=u^{q-1}\]
under Neumann boundary conditions, in the unit ball of $\mathbb R^N$.
We use a variational approach in an invariant cone. We distinguish the two solutions upon their energy: one is a ground state
inside a Nehari-type subset of the cone, the other is obtained via a mountain pass argument inside the Nehari set.

As a byproduct of our proofs, we detect the limit profile of the low energy solution as $q\to\infty$ and show that the constant solution 1 is a local minimum on the Nehari set.
\end{abstract}

\maketitle

\section{Introduction}
For $1<p<2$, we consider the following Neumann problem
\begin{equation}\label{eq:Pq}
\begin{cases}
-\Delta_p u+u^{p-1}=u^{q-1}\quad&\mbox{in }B,\\
u>0\quad&\mbox{in }B,\\
\partial_\nu u=0\quad&\mbox{on }\partial B,
\end{cases}
\end{equation}
where $B$ denotes the unit ball of $\mathbb R^N$ ($N\ge1$), $\nu$ is the outer unit normal of $\partial B$ and $q>p$. In particular, the nonlinearity on the right-hand side is allowed to be supercritical in the sense of Sobolev embeddings and, for $q$ sufficiently large, we prove that the problem admits two distinct nonconstant radial, radially nondecreasing solutions.

Although we address the problem governed by the --possibly singular-- $p$-Laplacian operator, with $p\in(1,2)$, the interest in this class of problems originally arose for the case $p=2$, as a stationary version of the Keller-Segel system for chemotaxis. For the semilinear problem, the existence and non-existence of nonconstant solutions has been widely studied since the eighties. In \cite{LNT}, in the subcritical regime, Lin, Ni and Takagi proved that if the radius of the ball is sufficiently small, the semilinear problem admits only the constant solution, while, if the radius is sufficiently large, it admits a nonconstant solution. Similar existence and non-existence results have been proved also in the supercritical regime in \cite{LN}. Conversely, the validity of such results in the critical case depends on the dimension $N$, cf. \cite{AY91,AY97,BKP}. More recently, for a general nonlinearity $f(u)$ on the right-hand side, it has been proved in \cite{BNW} that the semilinear problem admits a radial, radially increasing solution if $f(1)=1$ and the radial Morse index of the constant solution $u\equiv1$ is greater than one. When the nonlinearity is the pure power $f(u)=u^{q-1}$, the previous hypothesis on the radial Morse index reads as $q>2+\lambda_2^{\mathrm{rad}}(R)$, where $\lambda_2^{\mathrm{rad}}(R)$ is the first nonzero radial eigenvalue of the Laplacian in the ball $B(R)$, with Neumann boundary conditions. From this assumption, it is apparent that the existence of nonconstant solutions for this kind of problems is related to the radius of the ball or to the exponent $q$. Subsequently, for any $k\in\mathbb N$, under the analogous hypothesis $q>2+\lambda_{k+1}^{\mathrm{rad}}(R)$, the existence of $k$ oscillating radial solutions has been proved in \cite{BGT} via bifurcation techniques, in \cite{BGNT} via a perturbative approach and variational methods, and in \cite{ABF-ESAIM} using the shooting method for ODEs and a phase plane analysis.

For the quasilinear problem the situation is quite different and strongly depends on whether $p$ is greater or less than 2. A first non-existence result for the critical $p$-Laplacian problem with $p>2$, in a small ball, is contained in \cite{AY97}. Much more recently, by means of variational techniques, the existence of a nonconstant radial, nondecreasing solution has been proved in \cite{BF} in the case $p>2$, for every $q>p$, regardless of the radius of the ball. Even more, in \cite{ABF-ESAIM}, it has been proved that, if $q>p>2$, problem \eqref{eq:Pq} admits infinitely many nonconstant radial solutions. In the same paper, also the case $p<2$ has been considered, but the type of result is quite different: it is shown that for every $k\in\mathbb N$ there exists $R_k>0$ such that if the radius is greater than $R_k$, the problem admits $2k$ nonconstant radial solutions, which in couple share the same oscillatory behavior. In particular, for $R>R_1$, the existence of two increasing solutions is obtained via shooting approach, see also \cite{ABF-PRE} for solutions with reverse monotonicity properties in the subcritical case. On the other side, numerical simulations suggest that the existence of such solutions for $q$ large is independent of the radius of the ball. In Figure \ref{fig}, we represent the branch of radial, radially increasing solutions of \eqref{eq:Pq} when varying the parameter $q$. From the picture it is clear that, for a fixed value of $q$ sufficiently large, besides the constant solution $u\equiv1$, there are two more solutions on this branch. We refer to \cite[Section 3]{ABF-ESAIM} for further bifurcation diagrams and comparisons with the cases $p>2$ or $p=2$.

\begin{figure}\label{fig}
\includegraphics[scale=.5]{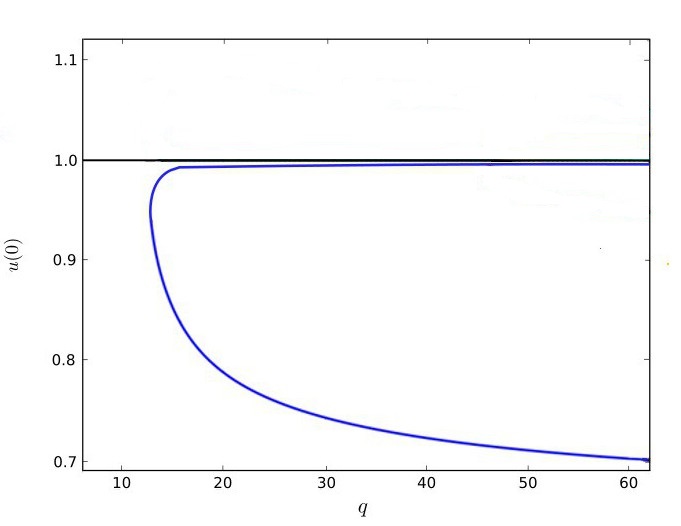}
\caption{In blue, the branch of radially nondecreasing solutions $u$ of \eqref{eq:Pq}, plotted as $u(0)$ as function of $q$. Both the upper and the lower parts of such a branch seem to persist for all values of $q$. Moreover, the blue branch do not bifurcate from the one of constant solutions $u\equiv1$. The figure is obtained numerically with the sotfware AUTO-07p \cite{AUTO}, for problem \eqref{eq:Pq} with $p=1.97$ in dimension $N=1$.}
\end{figure}

In the present paper, we obtain the existence of two increasing solutions, $u_q$ and $v_q$, under the assumption that the exponent $q$ is large enough, independently of the radius of the ball. The variational techniques applied here allow us to detect which of the two solutions has higher energy, and to identify the limit behavior as $q\to\infty$ of the one with lower energy. We further observe that the numerical simulations suggest that the higher energy solution $v_q$ should converge to the constant $1$ as $q\to\infty$, which is an interesting open problem.

In order to state rigorously our results, we introduce here some objects that will be used throughout the paper.

We work in the set
\begin{equation}\label{cone}
\mathcal C:=\{u\in W^{1,p}_{\mathrm{rad}}(B)\,:\, u\ge0,\,u(r_1)\le u(r_2) \mbox{ for all }0<r_1\le r_2\le1\},
\end{equation}
where with abuse of notation we write $u(|x|):=u(x)$.
This set is a closed convex cone in $W^{1,p}(B)$ and was first introduced in \cite{SerraTilli2011} in the context of a similar problem with $p=2$. Working in this cone has the twofold advantage of recovering the compactness in this supercritical regime, cf. Lemma \ref{bounded}, and of knowing a priori the monotonicity of the solutions that will be found therein.
On the other hand, since this cone has empty interior in the $W^{1,p}$-topology, see \cite[Introduction]{SerraTilli2011}, it is not possible to apply directly the Mountain Pass Theorem in $\mathcal C$: thanks to a priori estimates in the cone, we apply the truncation method and refine the Deformation Lemma to find a mountain pass solution of the problem inside the cone.

We introduce a Nehari-type set inside $\mathcal{C}$ as follows
\begin{equation*}
\mathcal N_q:=\left\{u\in\mathcal C\setminus\{0\}\,: \int_B(|\nabla u|^p+|u|^p)dx=\int_B f_q(u)u\,dx\right\},
\end{equation*}
where $f_q$ is a suitable truncated nonlinearity, that is Sobolev-subcritical (see Lemma \ref{truncated} below). Letting also $F_q(u):=\int_0^u f_q(s)ds$, we shall consider the following modified energy functional
\begin{equation*}
I_q(u):= \int_B\left(\frac{|\nabla u|^p}{p}+\frac{|u|^p}{p}-{F}_q(u)\right)dx.
\end{equation*}

The first result of the paper is the existence, for $q$ sufficiently large, of a nonconstant radial solution $u_q$ and the detection of its limit profile as $q\to\infty$.

\begin{theorem}\label{thm:asymptotic_q}
For $q$ sufficiently large there exists a nonconstant solution $u_q\in\mathcal C$ of \eqref{eq:Pq}, which has the following variational characterization
\begin{equation}\label{eq:variational_u_q}
I_q(u_q)=\inf_{u\in\mathcal N_q}I_q(u).
\end{equation}
Moreover, as $q\to \infty$,  $I_q(u_q) < I_q(1)$ and
\begin{equation}\label{uqgoestoG}
u_q \to G \textrm{ in } W^{1,p}(B) \cap C^{0,\nu}(\bar B)
\end{equation}
for any $\nu\in(0,1)$, where $G$ is the unique solution of
\begin{equation}\label{eqforG}
\begin{cases}-\Delta_pG+|G|^{p-2}G=0\quad&\mbox{in }B,\\
G=1&\mbox{on }\partial B.
\end{cases}
\end{equation}
\end{theorem}

An immediate consequence of the previous theorem is that, for $q$ large, $u_q\not\equiv 1$: it is enough to notice that the limit problem \eqref{eqforG} does not admit the constant solution and conclude using the convergence in \eqref{uqgoestoG}. The proof technique for detecting the limit profile is inspired by \cite{GrossiNoris}, cf. also \cite{BF} for the case $p>2$. Moreover, as expressed by \eqref{eq:variational_u_q}, the solution $u_q$ is the minimizer of the modified energy $I_q$ restricted to the Nehari set $\mathcal{N}_q$. Contrarily to what happens for problem \eqref{eq:Pq} in the case $p\geq 2$, in the present case the constant solution $1$ is also a local minimizer of $I_q$ on $\mathcal{N}_q$, although not a global one, for large values of $q$.

\begin{theorem}\label{thm:1minimizer}
For any $q>2$ there exist two constants $\delta_q \in (0,1)$ and $M_q>0$ such that for every $w\in \mathcal{N}_q$ with the property $\|w-1\|_{W^{1,p}(B)}\leq\delta_q$, it holds
\[
I_q(w)-I_q(1)\ge M_q\|w-1\|_{W^{1,p}(B)}^p.
\]
\end{theorem}

From one side, being $u_q$ and 1 both minimizers on the Nehari set, it is more difficult to distinguish them using a comparison between their energies. We notice in passing that, with respect to the case $p\ge 2$, an additional difficulty arises here due to the fact that functional $I_q$ is not of class $C^2$ for $p<2$. The key result to prove the previous theorem is Lemma \ref{lem:PW}, in which we show that a Poincar\'e-Wirtinger-type inequality holds in a neighborhood of 1. On the other side, the presence of two minimizers produces a third solution. Indeed, taking advantage of Theorem \ref{thm:1minimizer}, we can prove the existence, for $q$ sufficiently large, of another nonconstant solution $v_q$ of \eqref{eq:Pq}, corresponding to a mountain pass type solution over $\mathcal{N}_q$.

\begin{theorem}\label{thm:terza_sol}
For $q$ sufficiently large there exists another nonconstant solution $v_q\in\mathcal C$ of \eqref{eq:Pq},  distinct from $u_q$.
\end{theorem}

We observe that, being $\mathcal{N}_q$ the intersection between the Nehari manifold and the cone $\mathcal C$, it has no more the structure of a manifold. Therefore, also for this second solution, we cannot apply directly the standard theorems of Critical Point Theory. In this case, we need to define a candidate critical level in terms of two-dimensional paths, cf. definition \eqref{eq:d_q}, and then use again the refined version of the deformation lemma inside the cone $\mathcal C$. Compared with the shooting method used in \cite{ABF-ESAIM}, one of the advantages of this approach is that we know, by construction, that $v_q$ has higher energy than $u_q$.

\medskip

The paper is organized as follows. In Section \ref{sec:truncated} we establish some a priori bounds for the solutions of \eqref{eq:Pq} belonging to $\mathcal{C}$; this allows us to define a truncated nonlinearity that is Sobolev-subcritical.  In Section \ref{sec:existence_mp} we apply a mountain pass type theorem inside the cone $\mathcal{C}$, in order to prove the existence of a mountain pass solution of \eqref{eq:Pq}. In order to show that such solution is nonconstant for sufficiently large values of $q$, we detect its limiting behaviour as $q\to\infty$: this is done in Section \ref{sec:q_large}, where we can conclude the proof of Theorem \ref{thm:asymptotic_q}. We prove Theorem \ref{thm:1minimizer} in Section \ref{sec:1minimizer}. The property stated therein is the main ingredient for the proof of the existence of a third solution, that is concluded in Section \ref{sec:terza_sol}.

\section{A priori estimates and truncated problem}\label{sec:truncated}
In this section we establish some a priori estimates for the solutions of \eqref{eq:Pq} belonging to $\mathcal{C}$, and also for a slightly more general problem. Our aim is to truncate the nonlinearity $u^{q-1}$, in order to replace it with a Sobolev subcritical one, but keeping the same $\mathcal{C}$-solutions.

Since we are interested in the regime $q\to\infty$, in the following we take, for simplicity,
\[
q>2,
\]
in such a way that the nonlinearities involved are of class $C^1$ also in the origin.

Let us first recall some known properties of the set $\mathcal{C}$ defined in \eqref{cone}, since it will play a very important role in all the paper. We note that the definition of $\mathcal C$ is well-posed because $W^{1,p}_{\mathrm{rad}}(B)$-functions can be taken continuous in $(0,1]$. Moreover, by monotonicity, for every $u\in\mathcal C$, we can set $u(0):=\lim_{r\to0^+}u(r)$ and consider $u\in C(\bar{B})$. Finally, being nondecreasing, every $u\in\mathcal C$ is differentiable a.e. and $u'(r)\ge 0$ where it is defined.
As already mentioned in the Introduction, the set $\mathcal C$ is a closed convex cone in $W^{1,p}(B)$: for all $u,\,v\in\mathcal C$ and $\lambda\ge0$ the following properties hold
\begin{itemize}
\item[(i)] $\lambda u\in \mathcal C$;
\item[(ii)] $u+v\in \mathcal C$;
\item[(iii)] if also $-u\in\mathcal C$, then $u\equiv0$;
\item[(iv)] $\mathcal C$ is closed for the topology of $W^{1,p}$.
\end{itemize}

Working in the cone $\mathcal C$ allows us to treat supercritical nonlinearities thanks to the property stated in the following lemma.

\begin{lemma}[{\cite[Lemma 2.2]{BF}}]\label{bounded}
For every $t\in[1,\infty)$ there exists $C(N,t)$ such that
$$\|u\|_{L^\infty(B)}\le C(N,t)\|u\|_{W^{1,t}(B)}\quad\mbox{for all }u\in\mathcal{C}.$$
\end{lemma}

In particular, by applying Lemma \ref{bounded} with $t=p$, we obtain that
\begin{equation}\label{eq:C_embedded_Linf}
\mathcal{C} \subset L^\infty(B).
\end{equation}
Another consequence of Lemma \ref{bounded} is that the cone $\mathcal C$ endowed with the $W^{1,p}$-norm is compactly embedded in $L^t(B)$ for all $t\in[1,\infty)$, see {\cite[Lemma 2.3]{BF}} for details.

Let $p^*$ be the critical exponent for the Sobolev embedding $W^{1,p}(B)\hookrightarrow L^t(B)$, namely
\[p^*:=\begin{cases}
\frac{Np}{N-p}\quad&\mbox{if }N>p,\\
+\infty&\mbox{otherwise.}
\end{cases}
\]
Fix $\ell\in (p,p^*)$. We now consider a class of modified problems
\begin{equation}\label{Pphi}
\begin{cases}
-\Delta_p u+ u^{p-1}=\varphi(u)\quad&\mbox{in }B,\\
u>0&\mbox{in }B,\\
\partial_\nu u=0&\mbox{on }\partial B,
\end{cases}
\end{equation}
where $\varphi$ can be any function of the form
\begin{equation}\label{eq:phi_s0}
\varphi(s)=\varphi_{q,s_0}(s):=\begin{cases}s^{q-1}\quad&\mbox{if }s\in[0,s_0],\\
s_0^{q-1}+\frac{q-1}{\ell-1}s_0^{q-\ell}(s^{\ell-1}-s_0^{\ell-1})&\mbox{if } s\in(s_0,\infty),\end{cases}
\end{equation}
with $s_0\in (2,\infty)$.  Notice that the functions $\varphi$ are of class $C^1$, nonnegative and increasing. In the next lemma we prove that the solutions of \eqref{Pphi} belonging to $\mathcal{C}$ are bounded in the $C^1$-norm, independently of $q$ and $s_0$.

\begin{lemma}\label{lem:a_priori_bounds}
Every solution $u\in\mathcal C$ of \eqref{Pphi},  for every $\varphi$ of the form \eqref{eq:phi_s0},  satisfies
\[
\|u\|_{L^\infty(B)} \leq 1+(p')^{1/p}
\quad\mbox{and}\quad
\|u'\|_{L^\infty(B)} \leq (p')^{1/p}.
\]
\end{lemma}
\begin{proof}
Let $\varphi=\varphi_{q,s_0}$ be any function of the form \eqref{eq:phi_s0} and let $u\in \mathcal{C}$ be any solution of \eqref{Pphi}. We first show that
\begin{equation}\label{eq:u(0)<1}
u(0)\leq1.
\end{equation}
To this aim, suppose by contradiction that $u(0)>1$. By using equation \eqref{eq:phi_s0} and the fact that $u$ is nondecreasing, we obtain
\[
\int_{\{u\ge s_0\}}(u^{p-1}-\varphi(u))\,dx=
-\int_{\{u\le s_0\}}(u^{p-1}-u^{q-1})\,dx>0.
\]
This contradicts the fact that, for $u> s_0$,
\begin{multline*}
u^{p-1}-\varphi(u)
%u^{p-1}-s_0^{q-1}-\frac{q-1}{\ell-1}s_0^{q-\ell}(u^{\ell-1}-s_0^{\ell-1})
=u^{p-1}-\frac{q-1}{\ell-1}s_0^{q-\ell}u^{\ell-1}+\frac{q-\ell}{\ell-1}s_0^{q-1} \\
\leq u^{p-1}-\frac{q-1}{2(\ell-1)}s_0^{q-\ell}u^{\ell-1}+\frac{q-\ell}{2(\ell-1)}s_0^{q-1} <0,
\end{multline*}
where we used that $\varphi$ is nonnegative and, in the last line, we applied relation (62) in \cite{BF}, with $M=2$. Hence \eqref{eq:u(0)<1} is established.

Proceeding similarly to \cite[Lemma 2.2]{ABF-ESAIM}, we let, for any $u\geq 0$, $\Phi(u)=\int_1^u \varphi(s)\,ds$ and, for any $r\in [0,1]$,
\[
H(r):=\frac{u'(r)^p}{p'}+\Phi(u(r))
\]
with $1/p'=1-1/p$. By making use of the equation satisfied by $u$ in radial form, we conclude that
\[
H'(r)=-\frac{N-1}{r}u'(r)^p \leq 0 \quad \text{ for every } r\in (0,1].
\]
Being $\Phi\geq0$ and $u'(0)=0$, this implies
\[
\frac{u'(r)^p}{p'} \leq H(r) \leq H(0)=\Phi(u(0))=u(0)^{q-1} \leq 1
 \quad \text{ for every } r\in [0,1],
\]
where in the last step we used \eqref{eq:u(0)<1}. Consequently,
\[
u(r) = u(0) + \int_0^r u'(s) \,ds \leq 1+(p')^{1/p}
\]
for every $r\in [0,1]$.
\end{proof}

In the light of Lemma \ref{lem:a_priori_bounds}, we now choose a specific function $\varphi$ of the form \eqref{eq:phi_s0} in such a way that every solution of \eqref{Pphi} with this specific $\varphi$,  belonging to $\mathcal{C}$, solves also the original problem \eqref{eq:Pq}. To this aim, we choose $s_0$ greater both than the $L^\infty$ bound and than another constant that will be needed later; from now on we let
\begin{equation}\label{eq:s0}
s_0:=\max \left\{ 2+(p')^{1/p},  C(N,p)(1+|B|^{1/p}) \right\},
\end{equation}
$C(N,p)$ being the constant that appears in Lemma \ref{bounded}.

\begin{lemma}\label{truncated}
Define ${f}_q(s):=\varphi_{q,s_0}(s)$.
If $u\in\mathcal C$ solves
\begin{equation}\label{eq:f_tilde_q}
\begin{cases}
-\Delta_p u+ u^{p-1}={f}_q(u)\quad&\mbox{in }B,\\
u>0&\mbox{in }B,\\
\partial_\nu u=0&\mbox{on }\partial B,
\end{cases}
\end{equation}
then $u$ solves \eqref{eq:Pq}.
\end{lemma}
\begin{proof}
Let $u\in \mathcal{C}$ be a solution of \eqref{eq:f_tilde_q}. By Lemma \ref{lem:a_priori_bounds} it holds $\|u\|_{L^\infty(B)}\leq 1+(p')^{1/p}<s_0$, so that
\[
{f}_q(u(x))=u(x)^{q-1} \quad\text{for every }x\in B.
\]
Hence $u$ solves \eqref{eq:Pq}.
\end{proof}

\begin{remark}
By direct calculations one can check that ${f}_q$ is of class $C^1$, nonnegative, increasing and satisfies the following properties for every $q>\ell$:
\begin{equation}\label{eq:f/s_increasing}
\mbox{fixed any } s>0, \quad\mbox{the map }t\in(0,\infty)\mapsto\frac{{f}_q(ts)}{t^{p-1}} \textrm{ is increasing},
\end{equation}
\begin{equation}\label{eq:subcritical}
\lim_{s\to\infty} \frac{{f}_q(s)}{s^{\ell-1}}=\frac{q-1}{\ell-1}s_0^{q-\ell}.
\end{equation}
These two properties will play a role in the subsequent sections.
\end{remark}

\begin{remark}\label{rem:ftilde=f}
We notice, for future use, that for every $u\in \mathcal{C}$ satisfying $\|u-1\|_{W^{1,p}(B)}\leq 1$ it holds ${f}_q(u)=u^{q-1}$. Indeed, using Lemma \ref{bounded}, the triangular inequality and relation \eqref{eq:s0},
\[
\|u\|_{L^\infty(B)} \leq C(N,p) \|u\|_{W^{1,p}(B)} \leq C(N,p) (1+|B|^{1/p}) \leq s_0.
\]
\end{remark}

\section{Existence of a mountain pass radial solution}\label{sec:existence_mp}

The aim of this section is to prove the existence of a mountain pass type solution of \eqref{eq:Pq}. In view of Lemma \ref{truncated}, problems \eqref{eq:Pq} and \eqref{eq:f_tilde_q} have the same solutions in $\mathcal{C}$; the advantage of dealing with \eqref{eq:f_tilde_q} is that this problem is subcritical and it can be treated with variational methods. Nonetheless, being forced to work in the cone $\mathcal C$,  we cannot apply directly standard techniques, because $\mathcal{C}$ has empty interior in the $W^{1,p}$-topology.

From now on in the paper, $f_q$ is the function introduced in Lemma \ref{truncated}, extended to zero in $(-\infty,0)$. As already mentioned in the Introduction, denoting
$F_q(u):=\int_0^u f_q(s)ds$, the energy functional associated to problem \eqref{eq:f_tilde_q} is $I_q:W^{1,p}(B)\to\mathbb R$ defined as
\begin{equation}\label{I}
I(u):=\int_B\left(\frac{|\nabla u|^p}{p}+\frac{|u|^p}{p}-{F}_q(u)\right)dx.
\end{equation}
Being $\ell<p^*$ and thanks to the Sobolev embedding, the functional $I_q$ is well defined and of class $C^1$.
We can also associate to \eqref{eq:f_tilde_q} the Nehari-type set
\begin{equation}\label{eq:M_q_def}
\mathcal{N}_q:=\left\{u\in\mathcal C\setminus\{0\}\,: \int_B(|\nabla u|^p+|u|^p)dx=\int_B f_q(u)u\,dx\right\}.
\end{equation}
As problems \eqref{eq:Pq} and \eqref{eq:f_tilde_q} need not be equivalent outside $\mathcal C$, we define $\mathcal N_q$ as the intersection of the cone $\mathcal C$ with the standard Nehari manifold of \eqref{eq:f_tilde_q}; this destroys the structure of manifold for $\mathcal N_q$.
On the other hand, being $\mathcal{N}_q$ a subset of $\mathcal{C}$, it is embedded in $L^\infty(B)$, cf.  \eqref{eq:C_embedded_Linf}.  It is a standard property that Nehari sets are bounded away from the origin; in this setting, an additional feature is that such a bound is independent of $q$.

\begin{lemma}[{\cite[Lemma 5.2]{BF}}]\label{palla_nella_Nehari}
There exists $\sigma>0$ such that
\[
\inf_{q\ge p+1}\inf_{u\in\mathcal N_q}\|u\|_{L^\infty(B)}\ge\sigma.
\]
\end{lemma}

The functional $I_q$ satisfies the mountain pass geometry and the Palais-Smale condition:

\begin{lemma}\label{lem:geometry}
Let $\tau\in (0,\min\{\sigma, 1\})$, with $\sigma$ given in Lemma \ref{palla_nella_Nehari},
\begin{itemize}
\item[(i)] there exists $\alpha_q>0$ such that $I_q(u)\ge \alpha_q$ for every $u \in \mathcal C$ with $\|u\|_{L^\infty(B)}=\tau$;
\item[(ii)] there exists $k>\tau$ such that  $I_q(k\cdot 1)<0$.
\end{itemize}
\end{lemma}
\begin{proof}
The proof of part (i) is contained in {\cite[Lemma 3.9]{BF}}. We prove now part (ii).
By \eqref{eq:subcritical}, there exists $\tilde{s}>s_0$ such that $f_q(s)>\frac{q-1}{2(\ell-1)}s_0^{q-\ell} s^{\ell-1}$ for every $s>\tilde{s}$. Hence, for every $k>\tilde{s}$, we get
\[
\begin{aligned}
I_q(k\cdot 1)&=|B|\left(\frac{k^p}{p}-\int_0^{\tilde{s}} f_q(s)ds-\int_{\tilde{s}}^k f_q(s)ds\right)\\
&\le |B|\left(\frac{k^p}{p}-\tilde{s}\min_{[0,\tilde{s}]} f_q- \frac{d}{2}\frac{k^{\ell}}{\ell}+\frac{d}{2}\frac{\tilde{s}^\ell}{\ell}\right).
\end{aligned}
\]
So, being $\ell>p$, $\lim_{k\to +\infty}I_q(k\cdot 1)=-\infty$, which concludes the proof.
\end{proof}

\begin{lemma}[{\cite[Lemma A.4]{BF}}]
$I_q$ satisfies the Palais-Smale condition, i.e. every sequence $(u_n)\subset W^{1,p}(B)$ such that $(I_q(u_n))$ is bounded and $I_q'(u_n)\to0$ in $(W^{1,p}(B))'$ admits a convergent subsequence.
\end{lemma}

We are now ready to introduce the mountain pass level
\begin{equation}\label{minmax}
c_q=\inf_{\gamma\in\Gamma_q}\max_{t\in[0,1]} I_q(\gamma(t)),
\end{equation}
where
$$
\Gamma_q:=\left\{ \gamma\in C([0,1];\mathcal C)\ :\  \gamma(0) \in U_{q,-},\: \gamma(1) \in U_{q,+}\right\},
$$
and
\begin{equation}
\begin{aligned}
U_{q,-} &= \left\{u \in \mathcal C \::\: I_q(u)<\frac{\alpha_q}{2},\:
\|u\|_{L^\infty(B)} < \tau\right\},\\
U_{q,+}&= \left\{u \in \mathcal C \, :\, I_q(u)< 0,\, \|u\|_{L^\infty(B)}>\tau\right\},
\end{aligned}
\end{equation}
with $\tau$, and $\alpha_q$ as in Lemma \ref{lem:geometry}.

\begin{theorem}\label{mountainpass1}
The value $c_q$ defined in \eqref{minmax} is finite and there exists a critical point $u_q\in\mathcal C$ of $I_q$ with $I_q(u_q)=c_q$.
\end{theorem}

\bigskip

The proof of this theorem requires several tools, which we introduce in the following.

Given the operator $T:(W^{1,p}(B))'\to W^{1,p}(B)$ such that $T(w)=v$, where $v$ is the unique solution to
\begin{equation}\label{T}
\begin{cases}-\Delta_p v+|v|^{p-2}v=w\quad&\mbox{in }B,\\
\partial_\nu v=0&\mbox{on }\partial B,
\end{cases}
\end{equation}
we introduce $\tilde T:W^{1,p}(B)\to W^{1,p}(B)$ defined by
\begin{equation}\label{eq:tildeT_def}
\tilde T(u)=T(f_q(u)).
\end{equation}
Being $\ell<p^*$, $u\in W^{1,p}(B)$ implies $u\in L^\ell(B)$. Hence, by \eqref{eq:phi_s0}, $f_q(u)\in L^{\ell'}(B)\subset (W^{1,p}(B))'$ and $\tilde T$ is well defined; moreover, $\tilde T$ preserves the cone $\mathcal C$, that is a crucial property for our technique.

\begin{lemma}[{\cite[Lemma 3.4]{BF}}]\label{cononelcono1}
The operator $\tilde T$ defined in \eqref{eq:tildeT_def} satisfies
$\tilde T(\mathcal C)\subseteq \mathcal C$.
\end{lemma}

\begin{proposition}[{\cite[Proposition A.3]{BF}}]\label{Ttildecompact1}
The operator $\tilde T$ is compact. Furthermore, there exist two positive constants $a,\,b$ such that for all $u\in W^{1,p}(B)$ the following properties hold
\begin{equation}\label{J1}
\begin{aligned}
&I_q'(u)[u-\tilde T(u)]\ge a\|u-\tilde T(u)\|_{W^{1,p}(B)}^2(\|u\|+\|\tilde T(u)\|_{W^{1,p}(B)})^{p-2},\\
&\|I_q'(u)\|_{*}\le b\|u-\tilde T(u)\|_{W^{1,p}(B)}^{p-1},
\end{aligned}
\end{equation}
where $\|\cdot\|_{*}$ denotes the norm in the dual space $(W^{1,p}(B))'$.
\end{proposition}

We note that \eqref{J1} implies that the set $\{u\,:\,\tilde T(u)=u\}$ coincides with the set of critical points of $I_q$.

\begin{lemma}[{\cite[Lemma 5.2]{BL}}]\label{lem:propJ3}
For every $c\in\mathbb R$ there exists $\kappa_c>0$ such that
\[
\|u\|_{W^{1,p}(B)}+\|\tilde T(u)\|_{W^{1,p}(B)}\le \kappa_c(1+\|u-\tilde T(u)\|_{W^{1,p}(B)})
\]
for every $u\in W^{1,p}(B)$ with $I_q(u)\le c$.
\end{lemma}

As the operator $u-\tilde T(u)$ is not Lipschitz,  it can not be used as a generalized pseudogradient vector field for $I'_q(u)$.  To overcome this obstacle, we rely on the results proved in \cite{BL,BartschLiuWeth} that are reformulated in our framework in \cite{BF}.

\begin{lemma}[{\cite[Lemma A.5]{BF}}]\label{KTI1}
Let $W:=W^{1,p}(B)\setminus\{u\,:\,\tilde T(u)=u\}$. There exists a locally Lipschitz continuous operator $K: W\to W^{1,p}(B)$ satisfying the following properties:
\begin{itemize}
\item[(i)] $K(\mathcal C\cap W)\subset \mathcal C$;
\item[(ii)] $\frac12\|u-K(u)\|_{W^{1,p}(B)}\le\|u-\tilde T(u)\|_{W^{1,p}(B)}\le2\|u-K(u)\|_{W^{1,p}(B)}$ for all $u\in W$;
\item[(iii)] let $a>0$ be the constant given in Proposition \ref{Ttildecompact1}, then
$$I_q'(u)[u-K(u)]\ge\frac{a}2 \|u-\tilde T(u)\|_{W^{1,p}(B)}^2(\|u\|_{W^{1,p}(B)}+\|\tilde T(u)\|_{W^{1,p}(B)})^{p-2}\quad\mbox{for all }u\in W.$$
\end{itemize}
\end{lemma}

\begin{lemma}[{\cite[Lemma A.6]{BF}}]\label{conseqPS1}
Let $c\in\mathbb R$ be such that $I_q'(u)\neq 0$ for all $u\in \mathcal C$ with $I_q(u)=c$. Then there exist two positive constants $\bar\varepsilon$ and $\bar\delta$ such that the following inequalities hold
\begin{itemize}
\item[(i)] $\|I_q'(u)\|_*\ge\bar\delta$ for all $u\in \mathcal C$ with $|I_q(u)-c|\le 2\bar\varepsilon$;
\item[(ii)] $\|u-K(u)\|_{W^{1,p}(B)}\ge\bar\delta$ for all $u\in \mathcal C$ with $|I_q(u)-c|\le 2\bar\varepsilon$.
\end{itemize}
\end{lemma}

\begin{lemma}\label{deformation1} Let $c\in\mathbb R$ be such that $I_q'(u)\neq 0$ for all $u\in \mathcal C$ with $I_q(u)=c$ and let $\bar{\varepsilon}$ be as in Lemma \ref{conseqPS1}. For every $\varepsilon\leq\bar{\varepsilon}$ there exists a function $\eta:\mathcal C\to\mathcal C$ satisfying the following properties:
\begin{itemize}
\item[(i)] $\eta$ is continuous with respect to the topology of $W^{1,p}(B)$;
\item[(ii)] $I_q(\eta(u))\le I_q(u)$ for all $u\in\mathcal C$;
\item[(iii)] $I_q(\eta(u))\le c-\varepsilon$ for all $u\in\mathcal C$ such that $|I_q(u)-c|<\varepsilon$;
\item[(iv)] $\eta(u)=u$ for all $u\in\mathcal C$ such that $|I_q(u)-c|>2\varepsilon$.
\end{itemize}
\end{lemma}
\begin{proof}
Let $\varepsilon\leq\bar{\varepsilon}$.
Let $\chi:\mathbb R\to [0,1]$ be a smooth cut-off function such that
$$
\chi(t)=
\begin{cases}1\quad&\mbox{if }|t-c|<\varepsilon,\\
0&\mbox{if }|t-c|>2\varepsilon.
\end{cases}
$$
Recalling the definition of $K$ in Lemma \ref{KTI1}, let $\Phi: W^{1,p}(B)\to W^{1,p}(B)$ be the map defined by
$$\Phi(u):=\begin{cases}\chi(I_q(u))\frac{u-K(u)}{\|u-K(u)\|_{W^{1,p}(B)}}\quad&\mbox{if }|I_q(u)-c|\le 2\varepsilon,\\
0&\mbox{otherwise.}\end{cases}$$
Note that the definition of $\Phi$ is well posed by Lemma \ref{conseqPS1}. For all $u\in\mathcal C$, we consider the Cauchy problem
\begin{equation}\label{CauchyProblem}
\begin{cases}\frac{d}{dt}\eta(t,u(x))=-\Phi(\eta(t,u(x)))\quad&(t,x)\in(0,\infty)\times B,\\
\eta(0,u(x))=u(x) & x\in B.
 \end{cases}
\end{equation}
Being $K$ locally Lipschitz continuous by Lemma \ref{KTI1}, for all $u\in\mathcal C$ there exists a unique solution $\eta(\cdot, u)\in C^1([0,\infty);W^{1,p}(B))$.

For $u\in \mathcal C$, we shall define $\eta(u):=\eta(\bar t, u)$ for a suitable $\bar t$ to be specified later. Since for every $t$, $\eta(t,\cdot)$ preserves the cone and satisfies properties (i) and (iv), the same holds true also for $\eta$. In particular, the preservation of the cone can be proved as in \cite[Lemma 3.8]{BF}, using the property that  $\tilde T(\mathcal C)\subset \mathcal C$, cf. Lemma \ref{cononelcono1}.

Let us prove now (ii). Also in this case, it holds for all $t$. Indeed, for every $u\in\mathcal C$ and $t>0$, we can write
\begin{equation}\label{eq:flusso_decrescente1}
\begin{aligned}I_q(\eta(t,u))-I_q(u)&=\int_0^t\frac{d}{ds}I_q(\eta(s,u))ds\\
&\hspace{-2.5cm}=-\int_0^t\frac{\chi(I_q(\eta(s,u)))}{\|\eta(s,u)-K(\eta(s,u))\|_{W^{1,p}(B)}}I_q'(\eta(s,u))[\eta(s,u)-K(\eta(s,u))]ds\\
&\hspace{-2.5cm}\le-\frac{a}2\displaystyle{\int_0^t\dfrac{\|\eta(s,u)-\tilde T(\eta(s,u))\|_{W^{1,p}(B)}^2 \chi(I_q(\eta(s,u)))}{\|\eta(s,u)-K(\eta(s,u))\|_{W^{1,p}(B)}(\|\eta(s,u)\|_{W^{1,p}(B)}+\|\tilde T(\eta(s,u))\|_{W^{1,p}(B)})^{2-p}}}
ds\\
&\le0,
\end{aligned}
\end{equation}
where we have used the inequality in Lemma \ref{KTI1}-(iii).

It remains to choose $\bar t$ in such a way that (iii) holds. Let $u\in\mathcal C$ be such that $|I_q(u)-c|<\varepsilon$ and let $t$ be sufficiently large. Then, two cases arise: either there exists $s\in[0,t]$ for which $I_q(\eta(s,u))\le c-\varepsilon$ and so, by the previous calculation we get immediately that $I_q(\eta(t,u))\le c-\varepsilon$, or for all $s\in[0,t]$, $I_q(\eta(s,u))> c-\varepsilon$. In this second case,
$$c-\varepsilon< I_q(\eta(s,u))\le I_q(u)< c+\varepsilon.$$
In particular,  being $\varepsilon\leq\bar\varepsilon$,  Lemma \ref{conseqPS1}-(i) applies, providing $\eta(s,u)\in W:= W^{1,p}(B)\setminus\{u\,:\,\tilde T(u)=u\}$. By the definition of $\chi$ and Lemma \ref{conseqPS1}-(ii), it results that for all $s\in[0,t]$
$$
\chi(I_q(\eta(s,u)))=1,\quad\|\eta(s,u)-K(\eta(s,u))\|_{W^{1,p}(B)}\ge\bar\delta.
$$

Hence, by \eqref{eq:flusso_decrescente1}, Lemma \ref{lem:propJ3}, Lemma \ref{KTI1}-(ii)-(iii), and Lemma \ref{conseqPS1}, we obtain for all $t>0$
$$
\begin{aligned}
I_q(\eta(t,u))&\le I_q(u)\\
&\hspace{-1cm}-\frac{a}{2\kappa_{c+\bar\varepsilon}^{2-p}}\int_0^t\frac{\|\eta(s,u)-\tilde T(\eta(s,u))\|_{W^{1,p}(B)}^2}{\|\eta(s,u)-K(\eta(s,u))\|_{W^{1,p}(B)}(1+\|\eta(s,u)-\tilde T(\eta(s,u))\|_{W^{1,p}(B)})^{2-p}}ds\\
&\le I_q(u)-\frac{a\bar\delta}{8\kappa_{c+\varepsilon}^{2-p}}\int_0^t\frac{\|\eta(s,u)-K(\eta(s,u))\|_{W^{1,p}(B)}}{\left(1+\frac{1}{2}\|\eta(s,u)-K(\eta(s,u))\|_{W^{1,p}(B)}\right)^{2-p}}ds\\
&< c+\varepsilon- \frac{a\bar\delta}{8\kappa_{c+\varepsilon}^{2-p}}\left(1+\frac{\bar\delta}{2}\right)^{p-2}t,
\end{aligned}
$$
where we have used that the function $f(x)=x^2(1+x)^{p-2}$ is increasing in $\mathbb R^+$.

Therefore, $I_q(\eta(t,u))< c-\varepsilon$ for every
\[
t\ge\frac{16\varepsilon}{a\bar\delta}\left(1+\frac{\bar\delta}{2}\right)^{2-p}\kappa_{c+\varepsilon}^{2-p}=:\bar t,
\]
and the proof is concluded.
\end{proof}

\begin{proof}[$\bullet$ Proof of Theorem \ref{mountainpass1}]
By Lemma \ref{lem:geometry}, $\alpha_q\le c_q<\infty$. Indeed, for $k$ large enough, the curve $\gamma(t)=kt$, $t\in[0,1]$ belongs to $\Gamma$. Hence, $\Gamma\neq\emptyset$ and so $c_q<\infty$. On the other side, for every $\gamma\in\Gamma$, $\max_{[0,1]}I_q(\gamma) \ge \alpha_q$, and so also $c_q\ge \alpha_q$.

Now, suppose by contradiction that there are no critical points $u\in\mathcal C$ of $I_q$ at level $c_q$. By Lemma \ref{conseqPS1}-(i), $\|I'_q(u)\|_*\ge \bar\delta$ for every $u\in \mathcal C$ such that  $|I_q(u)-c_q|\le 2\bar\varepsilon$.
Fix
\[
\varepsilon < \min \left\{ \bar\varepsilon, \frac{c_q}{2}-\frac{\alpha_q}{4} \right\}.
\]

Let $\gamma\in\Gamma$ be any curve such that $\max_{t\in[0,1]}I_q(\gamma(t))<c_q+\varepsilon$, and define $\bar\gamma(t):=\eta(\gamma(t))$ for $t\in[0,1]$, with $\eta$ as in Lemma \ref{deformation1}.
Being $c_q-2\varepsilon>\alpha_q/2$, neither $\gamma(0)$ nor $\gamma(1)$ belong to the strip $\{u\,:\,|I_q(u)-c_q|\le 2\varepsilon\}$ and consequently, by Lemma \ref{deformation1}-(iv), $\bar\gamma\in\Gamma$. Hence, by Lemma \ref{deformation1}-(iii), $\max_{t\in[0,1]}I_q(\bar\gamma(t))<c_q$, contradicting the definition of $c_q$ as infimum.
\end{proof}

\section{The mountain pass solution is non-constant for $q$ large}\label{sec:q_large}
In this section we will find the limit profile, as $q\to\infty$, of the mountain pass solution $u_q$ whose existence has been proved in Theorem \ref{mountainpass1} for every $q>2$. As a byproduct of this result, we immediately have that $u_q$ is non-constant for $q$ large.

To this aim, we first state some lemmas whose proofs can be found in \cite[Section 5]{BF} for the case $p>2$, but they continue to hold also in this setting with $1<p<2$.

The next lemma ensures that the Nehari-type set $\mathcal{N}_q$ defined in \eqref{eq:M_q_def} is homeomorfic to a sphere; its proof uses property \eqref{eq:f/s_increasing} of ${f_q}$.

\begin{lemma}[cf. {\cite[Lemma 5.3]{BF}}]\label{gH} For every $u\in\mathcal C\setminus\{0\}$ there exists a unique $h_q(u)>0$ such that $h_q(u)u\in\mathcal{N}_q$. It holds
\begin{equation}\label{Iqhq}
I_q(tu)>0 \quad\mbox{for all }t\in(0,h_q(u)];
\end{equation}
\[
I'_q(tu)[u]>0 \text{ if and only if } t\in(0,h_q(u)).
\]
Furthermore, if $(u_n)\subset \mathcal C\setminus\{0\}$ is such that $u_n\to u\in \mathcal C\setminus\{0\}$ with respect to the $W^{1,p}$-norm, then $h_q(u_n)\to h_q(u)$. Finally, the map
$$H:u\in \mathcal C\cap\mathcal S^1\mapsto h_q(u)u\in\mathcal N_q,\quad\mbox{where }\mathcal S^1:= \{u\in W^{1,p}(B)\,:\,\|u\|_{W^{1,p}(B)}=1\}$$ is a homeomorphism.
\end{lemma}

The following lemma guarantees that the mountain pass level in the cone coincides with the infimum on the Nehari-type set in the cone.

\begin{lemma}[cf. {\cite[Lemma 5.4]{BF}}]\label{lemma:nehari}
The following equalities hold
\begin{equation}\label{c_uguali}
c_q=\inf_{u\in\mathcal C\setminus\{0\}}\sup_{t\ge0}I_q(tu)=\inf_{u\in\mathcal N_q}I_q(u).
\end{equation}
\end{lemma}

In the next lemma, we refine for $u_q$ the $C^1$-a priori bound previously obtained in Lemma \ref{lem:a_priori_bounds}, using that ${f}_q(u_q)=u_q^{q-1}$.

\begin{lemma}\label{commonbound} For all $q>2$,
$$\|u_q\|_{L^\infty(B)}\le \left(\frac{q}{p}\right)^{\frac{1}{q-p}}\quad\mbox{and}\quad \|u'_q\|_{L^\infty(B)}\le \left(\frac{q-p}{q(p-1)}\right)^{\frac{1}{p}}.$$
In particular, $\limsup_{q\to\infty}\|u_q\|_{L^\infty(B)}\le 1$ and  $\limsup_{q\to\infty}\|u'_q\|_{L^\infty(B)}\le p^{-\frac{1}{p}}$.
\end{lemma}
\begin{proof}
If $u_q\equiv1$ the thesis is immediately verified with $C=1$, otherwise the thesis can be proved following the argument in \cite[Lemma 5.5]{BF}.
\end{proof}

In view of the previous lemma, the existence of a limit profile follows.

\begin{lemma}[cf. {\cite[Lemma 5.6]{BF}}]\label{weakconv}
There exists a function $u_\infty\in \mathcal C$ for which
\begin{equation}\label{eq:convergences}
u_q\rightharpoonup u_\infty\;\mbox{ in }W^{1,p}(B), \quad u_q\to u_\infty \;\mbox{ in }C^{0,\nu}(\bar B) \quad\mbox{as }q\to\infty,
\end{equation}
for any $\nu\in(0,1)$. Furthermore, $u_\infty(1)=1$.
\end{lemma}
\begin{proof}
In view of Lemma \ref{commonbound} and using the compactness of the embedding $C^1\hookrightarrow C^{0,\nu}$, \eqref{eq:convergences} immediately follows. Furthermore, up to a subsequence $u_q\to u_\infty$ pointwise, hence $u_\infty\in\mathcal C$. As for the last part of the statement, we observe that, if $u_q\equiv 1$ for every $q$, then obviously $u_\infty\equiv 1$. Otherwise, if $u_q\not\equiv 1$, integrating the equation satisfied by $u_q$, we get
\[
\int_B u_q^{p-1}(1-u_q^{q-p}) dx =0.
\]
Since $u_q$ is positive and nondecreasing, we deduce that
\begin{equation}\label{eq:u(0)}
u_q(0)<1, \quad u_q(1)>1\quad\mbox{for all }q\ge 2>p.
\end{equation}
Hence, $\|u_q\|_{L^{\infty}(B)}=u_q(1)>1$. Consequently, together with Lemma \ref{commonbound}, we get
\[1\le \liminf_{q\to\infty}u_q(1)= \liminf_{q\to\infty}\|u_q\|_{L^\infty(B)}\le\limsup_{q\to\infty}\|u_q\|_{L^\infty(B)}\le 1,\]
and so $u_\infty(1)=\lim_{q\to\infty}\|u_q\|_{L^\infty(B)}=1$.
\end{proof}

We can give a variational characterization of the solution of \eqref{eqforG} and a relation with the mountain pass level $c_q$.

\begin{lemma}[cf. {\cite[Lemmas 5.7 and 5.8]{BF}}]\label{cinfinito}
The quantity
$$
c_\infty:=\inf\left\{\frac{\|v\|_{W^{1,p}(B)}^p}p\,:\,v\in\mathcal C,\,v=1\mbox{ on }\partial B\right\}
$$
is uniquely achieved by the radial function $G$ satisfying \eqref{eqforG}.
Furtheremore, the following chain of inequalities holds
\begin{equation}\label{eq:c_infty<c_q}
\frac{\|G\|_{W^{1,p}(B)}^p}{p}=c_\infty\le\frac{\|u_\infty\|_{W^{1,p}(B)}^p}{p}\le\liminf_{q\to\infty}\frac{\|u_q\|_{W^{1,p}(B)}^p}{p}\le \liminf_{q\to\infty}c_q.
\end{equation}
\end{lemma}

We are now ready to prove Theorem \ref{thm:asymptotic_q}.

\begin{proof}[$\bullet$ Proof of Theorem \ref{thm:asymptotic_q}.] Let $u_q$ be the mountain pass solution found in Theorem \ref{mountainpass1} for every $q>2$. We recall that $u_q$ can also be characterized by $I_q(u_q)=\inf_{u\in\mathcal N_q}I_q(u)$, cf. Lemma \ref{lemma:nehari}. Once we prove the asymptotic behavior \eqref{uqgoestoG} of $u_q$, being the convergence $C^{0,\nu}$ and recalling that $G$ is nonconstant, we can immediately conclude that $u_q$ is nonconstant for $q$ large enough.

The proof of \eqref{uqgoestoG} follows the lines of \cite[Theorem 1.3]{BF}; we report it here to highlight the role of the previous lemmas. Let $G$ be the unique solution of \eqref{eqforG}. Since $G\in\mathcal{C}\setminus\{0\}$, by Lemma \ref{gH} there exists a unique $h_q(G)>0$ such that $h_q(G)G \in \mathcal{N}_q$. We claim that, for $q$ large, $h_q(G)<s_0$. Indeed, let $t=\bar t_q$ be the unique solution of
\[
\|G\|_{W^{1,p}(B)}^p-\frac{1}{t^{p-1}}\int_B (tG)^{q-1}G dx=0,
\]
namely $\bar t_q=\left(\frac{\|G\|_{W^{1,p}(B)}^p}{\|G\|_{L^q(B)}^q}\right)^{\frac{1}{q-p}}$, then
\begin{equation}\label{eq:t_to1}
\bar t_q(G)\to\frac{1}{\|G\|_{L^\infty(B)}}=1 \quad\textrm{ as } q\to\infty.
\end{equation}
Since $s_0>2$, there exists $\bar q$ large, such that $\bar t_q<s_0$ for $q\ge \bar q$. Now, fix $q>\bar q$. Since by definition $t=h_q(G)$ is the unique solution of
\[
\|G\|_{W^{1,p}(B)}^p-\frac{1}{t^{p-1}}\int_B f_q(tG) G dx=0,
\]
and ${f}_q(s)=s^{q-1}$ as $s<s_0$, $h_q(G)=\bar t_q$ by uniqueness, and the claim is proved.

This implies that ${f}_q(h_q(G)G)=(h_q(G)G)^{q-1}$, and by \eqref{eq:t_to1} we get
\[
c_\infty=\frac{\|G\|_{W^{1,p}(B)}^p}{p}
=\lim_{q\to\infty} \frac{\|h_q(G) G\|_{W^{1,p}(B)}^p}{p}
=\lim_{q\to\infty} \left(I_q(h_q(G)G)+\frac{h_q(G)^q}{q}\int_B G^q dx \right).
\]
Using that $h_q(G)G\in\mathcal{N}_q$, we can rewrite the last term in the limit as follows
\[
c_\infty=\lim_{q\to\infty}\left( I_q(h_q(G)G)+ \frac{\|h_q(G) G\|_{W^{1,p}(B)}^p}{q} \right) = \lim_{q\to\infty} I_q(h_q(G)G).
\]
On the other hand, by Lemma \ref{lemma:nehari} we obtain
\[
c_q=\inf_{u\in\mathcal N_q}I_q(u)\le I_q(h_q(G)G).
\]
The previous two equations provide $c_\infty \geq \limsup_{q\to\infty} c_q$, which, together with Lemma~\ref{cinfinito}, imply
\begin{equation}\label{eq:c_infty_limit_c_q}
c_\infty=\lim_{q\to\infty}c_q.
\end{equation}
As a consequence, the inequalities in \eqref{eq:c_infty<c_q} are indeed equalities, so that
\[
\lim_{q\to\infty} \|u_q\|_{W^{1,p}(B)}=\|G\|_{W^{1,p}(B)}\quad\mbox{and}\quad \|u_\infty\|_{W^{1,p}(B)}=\|G\|_{W^{1,p}(B)}.
\]
Hence, $u_\infty$ achieves $c_\infty$ and, by Lemma \ref{cinfinito}, $u_\infty=G$.
Together with the $W^{1,p}$-weak convergence and the uniform convexity of $W^{1,p}(B)$, this implies that $u_q\to G$ in $W^{1,p}(B)$. By Lemma \ref{weakconv} the convergence is also $C^{0,\nu}(\bar B)$ for any $\nu\in(0,1)$.

Finally, let us prove that the following inequality holds
\begin{equation}
\label{eq:Iq<I1}
I_q(u_q)<I_q(1)\qquad\mbox{for $q$ sufficiently large.}
\end{equation}
Indeed, suppose by contradiction that there exists a sequence $q_n\to\infty$ such that $I_{q_n}(1)=c_{q_n}$ for every $n$.  As a consequence of \eqref{eq:c_infty_limit_c_q}, we can pass to the limit in the previous equality and obtain $\|1\|_{W^{1,p}(B)}^p/p=c_\infty$, thus contradicting the fact that $c_\infty$ is uniquely achieved by $G\not\equiv 1$, cf. Lemma \ref{cinfinito}.
\end{proof}

\section{The constant solution is a local minimizer on $\mathcal{N}_q$}\label{sec:1minimizer}

In this section we prove that the constant solution $1$ is a local minimizer on the Nehari-type set $\mathcal N_q$ for every $q>2$. To this aim, we shall need a Poincar\'e--Wirtinger-type inequality.

\begin{lemma}\label{lem:PW}
Fix $q>2$.
There exist $\delta_q \in (0,1)$ and a constant $C_{PW}>0$ such that for every $w\in \mathcal{N}_q$ with the property $\|w-1\|_{W^{1,p}(B)}\leq\delta_q$, it holds
\[
\|w-1\|_{L^p(B)}^p \leq C_{PW} \|\nabla w\|^p_{L^p(B)}.
\]
\end{lemma}
\begin{proof}
Suppose by contradiction that for every $n\in \mathbb{N}_q$ there exists $w_n \in \mathcal{N}_q$ such that
\begin{equation}\label{eq:contradiction_w_n}
\|w_n-1\|_{W^{1,p}(B)} \leq \frac{1}{n}
\quad\mbox{ and }\quad
\|w_n-1\|_{L^p(B)} \geq n \|\nabla w_n\|_{L^p(B)}.
\end{equation}
Letting
\[
\varepsilon_n:=\|w_n-1\|_{L^p(B)}
\quad\text{ and }\quad
v_n:=\frac{w_n-1}{\varepsilon_n},
\]
we have that $\|v_n\|_{L^p(B)} =1$ for every $n$ and that, by \eqref{eq:contradiction_w_n},  $\|\nabla v_n\|_{L^p(B)} \leq \frac{1}{n} \to 0$ as $n\to\infty$. Hence, there exist a subsequence $n_k$ and $\bar v\in W^{1,p}(B)$ such that,
\[
v_{n_k} \rightharpoonup \bar v \text{ weakly in } W^{1,p}(B)
\quad\text{ and }\quad
v_{n_k} \to \bar v \text{ in } L^p(B)
\]
as $k\to\infty$.
Moreover,
\begin{multline*}
\|\bar v\|_{W^{1,p}(B)}^p \leq
\liminf_{k\to\infty} \|v_{n_k}\|_{W^{1,p}(B)}^p \leq
\liminf_{k\to\infty} \left( \|v_{n_k}\|_{L^p(B)}^p+\frac{1}{n_k^p} \right) \\
=\|\bar v\|_{L^p(B)}^p \leq \|\bar v\|_{W^{1,p}(B)}^p,
\end{multline*}
providing that the convergence $v_{n_k} \to \bar v$ in $W^{1,p}(B)$ is actually strong and that $\bar v$ is a non-trivial constant function, more precisely $\bar v =\pm |B|^{-1/p}$.

We shall now exhibit a contradiction by exploiting the fact that $w_n \in \mathcal{N}_q$ for every $n$. Noticing that $f_q (w_n)=w_n^{q-1}$ for every $n$ (see Remark \ref{rem:ftilde=f}),  the Nehari condition for $w_n$ writes
\[
\int_B(|\nabla w_n|^p+w_n^p)dx=\int_B w_n^q\,dx.
\]
We rewrite the last equality in terms of $v_n$ and we divide it by $\varepsilon_n$ to obtain
\begin{equation}\label{eq:contradiction}
\varepsilon_n^{p-1} \int_B |\nabla v_n|^p dx=
\int_B w_n^p  \frac{(1+\varepsilon_n v_n)^{q-p}-1}{\varepsilon_n} \,dx,
\end{equation}
that readily leads to the contradiction $0=(q-p)\int_B \bar v \,dx$ by passing to the limit along the subsequence $n_k$ as $k\to\infty$. We remark that the converge of the right-hand side in \eqref{eq:contradiction} is justified by the Lebesgue dominated convergence theorem as
\begin{multline*}
\left| w_n^p(x)  \frac{(1+\varepsilon_n v_n(x))^{q-p}-1}{\varepsilon_n} \right|=
(q-p) \left| w_n^p(x)  (1+\xi_n(x) v_n(x))^{q-p-1} v_n(x) \right| \\
\leq (q-p) \|w_n\|_{L^\infty(B)}^{q-1} |v_n(x)|
\leq (q-p) s_0^{q-1} |v_n(x)|,
\end{multline*}
with $v_n\in L^1(B)$ and $\xi_n(x) \in (0,\varepsilon_n)$ given by the Lagrange theorem.
\end{proof}

\begin{remark}\label{rem:equivalence}
Given $\delta_q$ as in Lemma \ref{lem:PW}, for every $w\in \mathcal{N}_q$ with the property that $\|w-1\|_{W^{1,p}(B)}\leq\delta_q$ it holds
\[
\|\nabla w\|^p_{L^p(B)} \leq
\|w-1\|_{W^{1,p}(B)}^p \leq (C_{PW}+1) \|\nabla w\|^p_{L^p(B)} .
\]
\end{remark}

\begin{proof}[$\bullet$ Proof of Theorem \ref{thm:1minimizer}]
Let $w$ be as in the hypotheses and 
\[
g_t(s):=\frac{1}{t}\int_B(1+s(w-1))^tdx,
\] 
for $t>1$ and $s\in[0,1]$. Since by Remark \ref{rem:ftilde=f}, $f_q(w)=w^{q-1}$, we can write
\begin{equation}\label{eq:Iq-I1}
I_q(w)-I_q(1)=\frac{1}{p}\int_B|\nabla w|^pdx+(g_p(1)-g_p(0))-(g_q(1)-g_q(0)).
\end{equation}
Now, being $p>1$, the function $|x|^p$ is convex, and the inequality $(1+x)^p\ge 1+px$ holds for every $x\ge -1$. Thus, applying this inequality to $w-1$, we get
\[(1+(w-1))^p\ge 1+p(w-1),\]
which, integrated over $B$, becomes
\[
g_p(1)-g_p(0)\ge g'_p(0).
\]
On the other hand, since $q>2$, the function $g_q$ is of class $C^2$, and so we can write the following Taylor expansion
\[g_q(1)-g_q(0)=g'_q(0)+\frac{(q-1)}{2}\int_B(1+\xi(w-1))^{q-2}(w-1)^2dx\quad\mbox{for some }\xi\in(0,1).\]
We further observe that $g_p'(0)=g_q'(0)$. Therefore, combining together the previous consideration, we can estimate \eqref{eq:Iq-I1} as follows:
\begin{equation}\label{eq:estimateIq-I1}
\begin{aligned}
I_q(w)-I_q(1)&\ge \frac{1}{p}\int_B|\nabla w|^pdx-\frac{q-1}{2}\int_B(1+\xi(w-1))^{q-2}(w-1)^2dx\\
&\ge \frac{1}{p}\int_B|\nabla w|^pdx-\frac{q-1}{2}\max\{1,\|w\|_{L^\infty(B)}\}^{q-2}\int_B(w-1)^2dx\\
&\ge \frac{1}{p}\int_B|\nabla w|^pdx-\frac{q-1}{2}s_0^{q-2}\|w-1\|^2_{L^2(B)}.
\end{aligned}
\end{equation}
We distinguish now two cases, depending on whether the critical Sobolev exponent $p^*$ is greater or less than $2$.\smallskip

\noindent\underline{\it Case 1}: $p\ge \frac{2N}{N+2}$. In this case, $p^*\ge 2$. Hence, taking $\delta_q$ smaller if necessary and using Remark \ref{rem:equivalence}, we obtain by \eqref{eq:estimateIq-I1}
\begin{equation}\label{eq:est-diff}
I_q(w)-I_q(1)\ge C\|w-1\|^p_{W^{1,p}(B)}-C'\|w-1\|^2_{W^{1,p}(B)},
\end{equation}
where $C:=\frac{1}{p(C_{PW}+1)}>0$, $C':=\frac{q-1}{2}s_0^{q-2}C_S$, and $C_S$ arises from the Sobolev inequality for the embedding $W^{1,p}(B)\hookrightarrow L^2(B)$. Recalling that $p<2$ and that $\|w-1\|_{W^{1,p}(B)}\le \delta_q$, this estimate provides
\[
I_q(w)-I_q(1)\ge \frac{C}{2}\|w-1\|^p_{W^{1,p}(B)}\quad\mbox{for }\delta_q\ll 1.
\]

\noindent\underline{\it Case 2}: $p< \frac{2N}{N+2}$. In this case, $p^*< 2$. By the inequality in Remark \ref{rem:ftilde=f} and the triangle inequality, $\|w-1\|_{L^\infty(B)}\le s_0+1$. Therefore, we have
\[
\int_B (w-1)^2dx\le (s_0+1)^{2-p^*}\int_{B}(w-1)^{p^*}dx,
\]
and arguing as in the previous case, we get
\[
I_q(w)-I_q(1)\ge C\|w-1\|^p_{W^{1,p}(B)}-C''\|w-1\|^{p^*}_{W^{1,p}(B)},
\]
where $C'':=\frac{q-1}{2}s_0^{q-2}(s_0+1)^{2-p^*}C'_S$, and now $C'_S$ arises from the embedding $W^{1,p}(B)\hookrightarrow L^{p^*}(B)$. Again, for $\delta_q\ll 1$, this allows to conclude, being $p^*<p$.
\end{proof}
%\begin{remark}
%\textcolor{blue}{We observe that the constant $\delta_q$ in Theorem \ref{thm:1minimizer} tends to zero as $q\to\infty$. Indeed, in Case 1. of the previous proof, to get that the nonnegativity of right-hand side of \eqref{eq:est-diff}, one has to take
%\[
%\delta_q\le\left(\frac{2}{(q-1)p(C_{\mathrm{PW}}+1)s_0^{q-2}C_S}\right)^{\frac{1}{2-p}}\to 0\quad\mbox{as }q\to\infty.
%\]
%A similar calculation leads to the same conclusion in Case 2. This seems to be coherent with the numerical simulations, which suggest that the higher energy solution $v_q$ converges to 1 as $q\to\infty$.
%}
%\end{remark}

\begin{remark}\label{rem:strict}
Theorem \ref{thm:1minimizer} actually implies that the constant solution 1 is a {\it strict} local minimizer for $I_q$ on $\mathcal{N}_q$. Indeed, given $\delta_q$ as in Theorem \ref{thm:1minimizer}, if $w\in\mathcal N_q$ is such that $\|w-1\|_{W^{1,p}(B)}=\delta_q$, then
\begin{equation}\label{eq:strict}
I_q(w)\ge I_q(1)+M_q\delta_q^p.
\end{equation}
\end{remark}

\begin{remark}\label{rem:u_q_far_1}
Another consequence of Theorem \ref{thm:1minimizer} is that $u_q$ cannot be too close to $1$ for $q$ sufficiently large.  More precisely, $\|u_q-1\|_{W^{1,p}(B)}\geq\delta_q$.
Indeed, suppose by contradiction that $\|u_q-1\|_{W^{1,p}(B)}< \delta_q$. Since, by Theorem \ref{thm:asymptotic_q}, $u_q\not\equiv 1$ for $q$ sufficiently large, Theorem \ref{thm:1minimizer} would provide $I_q(u_q)\ge I_q(1)$, which contradicts \eqref{eq:Iq<I1}.
\end{remark}

\section{Existence of a higher energy nonconstant solution}\label{sec:terza_sol}

In order to prove the existence, for $q$ sufficiently large, of the second solution $v_q$, we shall apply a variational method over the Nehari set $\mathcal{N}_q$.  A mountain pass type theorem over $\mathcal{N}_q$ applies due to the fact that, as shown in the previous section, both the nonconstant solution $u_q$ and the constant solution $1$ are local minimizers of the energy $I_q$ over $\mathcal{N}_q$.  The main difficulty in what follows is that the Nehari set is not a manifold, which prevents us from directly applying the mountain pass theorem over manifolds. We shall instead construct a candidate critical level by means of two-dimensional paths and show that it is indeed critical using the deformation previously introduced (see Lemma \ref{deformation1}).
As the deformation takes place inside the cone $\mathcal{C}$,  we need to define a variational structure inside $\mathcal{C}$ itself; this is done keeping in mind the structure of the Nehari set in $\mathcal{C}$, see Lemma \ref{gH}.
Let us start with some preliminary estimates.

\begin{lemma}\label{lem:qR}
Let
\[
\gamma_0(t,s)=t(s u_q+1-s)
\]
There exists $\bar{q}>2$, such that, for every $q\geq \bar{q}$, there exist $0<R_1 \ll 1$, $R_2\gg1$ such that
\begin{itemize}
\item[(i)] $0<I_q(\gamma_0(R_1,s))\leq \frac{c_q}{2}$;
\item[(ii)] $I_q'(\gamma_0(R_1,s)) [\gamma_0(R_1,s)]>0$;
\item[(iii)] $I_q(\gamma_0(R_2,s)) <0 $;
\item[(iv)] $I_q'(\gamma_0(R_2,s)) [\gamma_0(R_2,s)]<0$
\end{itemize}
for every $s\in [0,1]$.
\end{lemma}
\begin{proof}
Notice first that, thanks to Lemma \ref{weakconv}, for every $s\in [0,1]$ and $t>0$ it holds
\begin{equation}\label{eq:gamma_0_infty_norm}
\|\gamma_0(t,s)\|_{L^\infty(B)} \leq
t \| u_q \|_{L^\infty(B)} = t u_q(1) \to t \quad \text{ as }q\to\infty.
\end{equation}
Hence it is possible to choose $R_1$ so small that $\|\gamma_0(R_1,s)\|_{L^\infty(B)} \le\tau < s_0$, where $\tau$ is given in Lemma \ref{lem:geometry}. This implies both $I_q(\gamma_0(R_1,s))>0$ and
\begin{equation}\label{eq:gamma_0_tilde_f}
{f}_q(\gamma_0(R_1,s))=\gamma_0(R_1,s)^{q-1}
\end{equation}
for every $s\in [0,1]$ and for every sufficiently large $q$. With similar estimates we obtain, for every $t>0$ and $s\in [0,1]$,
\[
I_q(\gamma_0(t,s))\leq \frac{\|\gamma_0(t,s)\|^p_{W^{1,p}(B)}}{p}
\to \frac{t^p}{p} \left( \|\nabla G\|^p_{L^p(B)} +|B|\right)
\quad \text{ as }q\to\infty,
\]
where we used the convergence proved in Theorem \ref{thm:asymptotic_q}. As $\lim_{q\to\infty} c_q=c_\infty>0$ (see \eqref{eq:c_infty_limit_c_q}), it is possible to choose $t=R_1$ so small that (i) holds for every sufficiently large $q$.  To prove (ii) we make use of \eqref{eq:gamma_0_tilde_f}:
\begin{multline*}
I_q'(\gamma_0(R_1,s)) [\gamma_0(R_1,s)] =
\|\gamma_0(R_1,s)\|^p_{W^{1,p}(B)}-\|\gamma_0(R_1,s)\|_{L^q(B)}^q
\\
\geq \|\gamma_0(R_1,s)\|^p_{L^p(B)}-|B| \|\gamma_0(R_1,s)\|^q_{L^\infty(B)}
\geq |B| \left( R_1^p u_q(0)^p - R_1^q u_q(1)^q \right) \\
\to |B| R_1^p G(0)^p \quad \text{ as }q\to\infty,
\end{multline*}
for every $s\in [0,1]$, where we used again Theorem \ref{thm:asymptotic_q}.  Therefore, taking $R_1$ smaller, if necessary, also (ii) holds for every sufficiently large $q$.

Let us now consider $R_2\gg1$. For every $s\in [0,1]$ and $t>0$ we have
\[
\inf_B \gamma_0(t,s) \geq t u_q(0) \to t G(0)
\quad \text{ as }q\to\infty.
\]
Hence, being $G(0)>0$ by Weak and Strong Maximum Principles \cite[Theorem 1.1]{Damascelli} and \cite[Theorem 5]{Vazquez}, it is possible to choose $R_2$ large enough that $\inf_B  \gamma_0(R_2,s) > s_0 $ for every $s\in [0,1]$ and $t>0$,  implying
\[
{f}_q(\gamma_0(R_2,s)) = -\frac{q-\ell}{\ell-1}s_0^{q-1} + \frac{q-1}{\ell-1} s_0^{q-\ell} (\gamma_0(R_2,s))^{\ell-1}
\]
for every $s\in [0,1]$ and sufficiently large $q$. This allows to prove that, for every $s\in [0,1]$,
\[
I_q(\gamma_0(R_2,s)) \leq -C_1 R_2^\ell+C_1 R_2^p+C_3 R_2+C_4,
\]
for some constants $C_1>0$ and $C_2,\,C_3,\,C_4\in\mathbb R$ independent of $R_2$. As $\ell>p$, (iii) holds for sufficiently large $R_2$, independent of $q$. The proof of (iv) is very similar to that of (iii).
\end{proof}

In what follows, we fix $q\geq \bar{q}$, with $\bar{q}$ given in Lemma \ref{lem:qR}. Let $R_1,R_2$ as in Lemma \ref{lem:qR}, we define $Q:=[R_1,R_2]\times [0,1] \subset \mathbb{R}^2$,
\[
\gamma_0:Q\to \mathcal{C},
\qquad
\gamma_0(t,s):=t(su_q+1-s).
\]
Notice that $\gamma_0(t,s) \in \mathcal{C}$ for every $(t,s)\in Q$ thanks to the convexity of $\mathcal{C}$. In particular, $\gamma_0$ belongs to the set
\[
\Upsilon_q:=\{\gamma \in C(Q;\mathcal C)\ :\ \gamma=\gamma_0 \text{ on } \partial Q \}.
\]
We define our candidate critical level as
\begin{equation}\label{eq:d_q}
d_q:=\inf_{\gamma\in\Upsilon_q} \max_{(t,s)\in Q} I_q(\gamma(t,s)).
\end{equation}

\begin{remark}\label{rem:boundary_Q}
The estimates proved in Lemma \ref{lem:qR} allow to conclude that,  for every $\gamma\in \Upsilon_q$,
\[
\max_{(t,s)\in \partial Q} I_q(\gamma(t,s)) = I_q(1).
\]
Indeed, notice first that,  being $\gamma=\gamma_0$ on $\partial Q$,  it is sufficient to estimate $I_q(\gamma_0(t,s))$,  $(t,s)\in \partial Q$. Lemma \ref{lem:qR} provides
\[
I_q(\gamma_0(R_2,s))<0<I_q(\gamma_0(R_1,s)) \leq \frac{c_q}{2} =\frac{I_q(u_q)}{2} \leq \frac{I_q(1)}{2}.
\]
Concerning the remaining part of $\partial Q$, we have for every $t>0$
\[
I_q(\gamma_0(t,0)) =I_q(t) \leq I_q(1)
\]
\[
I_q(\gamma_0(t,1)) =I_q(tu_q) \leq I_q(1)
\]
by the fact that $1,u_q \in \mathcal{N}_q$ and by Lemma \ref{gH}.
\end{remark}

\begin{lemma}\label{lem:d_q_estimates}
Given $M_q$ and $\delta_q$ as in Theorem \ref{thm:1minimizer}, it holds
\[
d_q \geq I_q(1) + M_q\delta_q^p.
\]
\end{lemma}
\begin{proof}
Given $\gamma\in\Upsilon_q$, we claim that there exists $(\bar{t},\bar{s})\in (R_1,R_2)\times(0,1)$ such that
\begin{equation}\label{eq:claim}
\gamma(\bar{t},\bar{s})\in\mathcal{N}_q
\quad\text{and}\quad
\|\gamma(\bar{t},\bar{s})-1\|_{W^{1,p}(B)}=\delta_q
\end{equation}
with $\delta_q$ as in Theorem \ref{thm:1minimizer}. Once the claim is proved,  Remark \ref{rem:strict} implies
\[
\max_{(t,s)\in Q} I_q(\gamma(t,s))
\geq I_q(\gamma(\bar{t},\bar{s}))
\geq I_q(1)+M_q \delta_q^p
\]
for every $\gamma\in\Upsilon_q$ and hence the statement.

In order to prove \eqref{eq:claim}, let $\mathcal{F},\mathcal{G}:\to\mathbb{R}$ be defined as
\[
\mathcal{F}(t,s)=\|h_q(\gamma(t,s))\gamma(t,s)-1\|_{W^{1,p}(B)} -\delta_q,
\]
with $h_q$ as in Lemma \ref{gH}, and
\[
\mathcal{G}(t,s)=I'_q(\gamma(t,s))[\gamma(t,s)].
\]
Notice that $\mathcal{F},\mathcal{G}$ are continuous in $Q$. Moreover, by Lemma \ref{lem:qR},
\[
\mathcal{G}(R_2,s)<0<\mathcal{G}(R_1,s)
\]
for every $s\in [0,1]$. By Lemma \ref{gH} and the fact that $1\in \mathcal{N}_q$, it holds
\[
\mathcal{F}(t,0)=\|h_q(t)t-1\|_{W^{1,p}(B)} -\delta_q = -\delta_q <0
\]
for every $t\in [R_1,R_2]$,  whereas, by Remark \ref{rem:u_q_far_1},
\[
\mathcal{F}(t,1)=\|h_q(tu_q)tu_q-1\|_{W^{1,p}(B)} -\delta_q
=\|u_q-1\|_{W^{1,p}(B)} -\delta_q >0
\]
for every $t\in [R_1,R_2]$. Hence, by Miranda's Theorem \cite{Miranda} there exists $(\bar{t},\bar{s})\in (R_1,R_2)\times(0,1)$ such that
\[
\mathcal{F}(\bar{t},\bar{s})=\mathcal{G}(\bar{t},\bar{s})=0,
\]
thus implying \eqref{eq:claim}.
\end{proof}

\begin{proof}[$\bullet$ Proof of Theorem \ref{thm:terza_sol}]
By Lemma \ref{lem:d_q_estimates} and the fact that $\Upsilon_q$ is not empty, for $q\ge \bar{q}$ we have that $d_q\in (I_q(1), \infty)$.
We need to show that, for sufficiently large $q$, $d_q$ is a critical level for $I_q$ in $\mathcal{C}$. To this aim we proceed by contradiction, thus assuming that there are no critical points $u\in\mathcal C$ of $I_q$ at level $d_q$.
Given $\bar{\varepsilon}$ as in Lemma \ref{conseqPS1},  let
\[
\varepsilon < \min \left\{ \bar\varepsilon, \frac{M_q\delta_q^p}{3} \right\},
\]
with $M_q,\delta_q$ as in Theorem \ref{thm:1minimizer}. By Lemma \ref{deformation1} there exists $\eta:\mathcal{C}\to\mathcal{C}$ such that $I_q(\eta(u))\le d_q-\varepsilon$ for all $u\in\mathcal C$ such that $|I_q(u)-d_q|<\varepsilon$ and $\eta(u)=u$ for all $u\in\mathcal C$ such that $|I_q(u)-d_q|>2\varepsilon$.

Let $\gamma\in \Upsilon_q$ be any path such that
\[
\max_{(t,s)\in Q} I_q(\gamma(t,s)) < d_q+\varepsilon.
\]
Notice that, by Remark \ref{rem:boundary_Q}, Lemma \ref{lem:d_q_estimates} and the choice of $\varepsilon$, we have
\[
\max_{(t,s)\in \partial Q} I_q(\gamma(t,s))=
\max_{(t,s)\in \partial Q} I_q(\gamma_0(t,s)) =I_q(1)
\leq d_q-M_q\delta_q^p \leq d_q-3\varepsilon.
\]
Therefore, defining $\bar\gamma(t,s):=\eta(\gamma(t,s))$ for $(t,s)\in Q$, we have that $\bar{\gamma}=\gamma$ on $\partial Q$ and thus $\bar{\gamma} \in \Upsilon_q$. Hence, by Lemma \ref{deformation1}-(iii),
\[
\max_{(t,s)\in \partial Q} I_q(\bar{\gamma}(t,s)) \leq d_q-\varepsilon
\]
thus contradicting the definition of $d_q$. This proves the existence of a critical point $v_q\in \mathcal C$ at level $d_q$. Since $I_q(v_q)>I_q(1)>I_q(u_q)$, the functions $v_q$, 1, $u_q$ are three distinct radially nondecreasing solutions of \eqref{eq:Pq}.
\end{proof}

\section*{Acknowledgments}
The first two authors were partially supported by the INdAM - GNAMPA Project 2020 ``Problemi ai limiti per l'equazione della curvatura media prescritta''. The last author was partially supported by the project Vain-Hopes within
the program VALERE - Universit\`a degli Studi della Campania ``Luigi Vanvitelli'', by the Portuguese
government through FCT/Portugal under the project PTDC/MAT-PUR/1788/2020, and by the INdAM - GNAMPA group.

\noindent

\bibliographystyle{abbrv}
\bibliography{biblio}

\end{document}